\documentclass{article}
\usepackage{amsmath}
\usepackage{amsthm}
\usepackage{amssymb,epsfig}
\usepackage[cp 1250]{inputenc}
\usepackage{a4}
\usepackage{graphics}
\newtheorem{definition}{Definition}
\newtheorem{lemma}{Lemma}
\newtheorem{theorem}{Theorem}
\newtheorem{example}{Example}
\newtheorem{corollary}{Corollary}
\newtheorem{observation}{Observation}

\begin{document}
\title{Continued Fractions of Quadratic Numbers}
\author{L\!'ubomíra Balková, Aranka Hrušková}

%
\maketitle
\begin{abstract}
In this paper, we will first summarize known results concerning continued fractions.
Then we will limit our consideration to continued fractions of quadratic numbers.
The second author described periods and sometimes precise form of continued fractions of $\sqrt{N}$, where $N$
is a natural number.
In cases where we were able to find such results in literature,
we recall the original authors, however many results seem to be new.
\end{abstract}
\section{Introduction}
Continued fractions have a~long history behind them - their origin may be placed to the age of Euclid's algorithm for the greatest common divisor or even earlier. However, they experience a~revival nowadays thanks to their applications in high-speed and high-accuracy computer arithmetics.
Some of the advantages of continued fractions in computer arithmetics are: faster division and multiplication than with positional number representations, fast and precise evaluation of trigonometric, logarithmic and other functions, precise representation of transcendental numbers, no roundoff or truncation errors (\cite{Se}, Kahan's method in~\cite{Mu} page 179).
\section{Continued fractions}
Here are some basic definitions and results that can be found in any number theory course~\cite{Kl, MaPe, La}.
We use $\lfloor x \rfloor$ to denote the integer part of a~real number $x$.
\begin{definition}
The continued fraction (expansion) of a real number $x$ is the sequence of integers $(a_n)_{n \in \mathbb N}$ obtained by the following algorithm
$$x_0=x, \quad a_n=\lfloor x_n\rfloor, \quad x_{n+1}=\left\{\begin{array}{l}\frac{1}{x_n-a_n} \ \text{if $x_n \not \in \mathbb Z$},\\
0 \ \text{otherwise.}\end{array}\right.$$
\end{definition}
Note that $a_0 \in \mathbb Z$ and $a_n \in \mathbb N$.
The algorithm producing the continued fraction is closely related to the Euclidean algorithm for computing the greatest common divisor of two integers. It is thus readily seen that if the number $x$ is rational, the algorithm eventually produces zeroes, i.e. there exists $N \in \mathbb N$ such that $a_n=0$ for all $n> N$, thus
\begin{equation} \label{finite_CF}
x=a_0+\cfrac{1}{a_1+\cfrac{1}{a_2+\cfrac{1}{\ddots+\cfrac{1}{{a_{N-1}+\cfrac{1}{a_N}}}}}}\ .
\end{equation}

We write $x=[a_0, \dots, a_N]$.
On the other hand, if we want to find an expression of the form~\eqref{finite_CF} with $a_0 \in \mathbb Z$ and
$a_n \in \mathbb N\setminus \{0\}$ otherwise, then there are exactly two of them -- the continued fraction $[a_0, \dots, a_N]$ and
$$
x=a_0+\cfrac{1}{a_1+\cfrac{1}{a_2+\cfrac{1}{\ddots+\cfrac{1}{{a_{N-1}+\cfrac{1}{(a_N-1)+\cfrac{1}{1}}}}}}}\ .
$$

If the number $x$ is irrational, then the sequence of the so-called {\em convergents}
\begin{equation}\label{convergents}
a_0,\ a_0+\cfrac{1}{a_1},\ \dots, a_0+\cfrac{1}{a_1+\cfrac{1}{a_2+\cfrac{1}{\ddots+\cfrac{1}{{a_{n-1}+\cfrac{1}{a_n}}}}}}
\end{equation}
converges to $x$ for $n \to \infty$.
On the other hand, every sequence of rational numbers of the form~\eqref{convergents} with $a_0 \in \mathbb Z$ and $a_n \in \mathbb N\setminus \{0\}$ converges to an irrational number (and for every irrational number there is only one such sequence -- the sequence of convergents).
We write $x=[a_0, \dots, a_n, \dots]$.

The convergents of the continued fraction are known to represent irrational numbers better than any other fractions.
\begin{theorem}[Lagrange]
Let $x \in \mathbb R \setminus \mathbb Q$ and let $\frac{p_n}{q_n}$ be its $n$-th convergent (where $p_n$ and $q_n$ are coprime) and let $\frac{p}{q}$ with $p,q \in \mathbb Z$ be distinct from $\frac{p_n}{q_n}$ and such that $0 < q\leq q_n$. Then $$\left|x-\frac{p_n}{q_n}\right|<\left|x-\frac{p}{q}\right|.$$
\end{theorem}
It is also known how well the continued fractions approximate irrational numbers.
\begin{theorem}
Let $x \in \mathbb R \setminus \mathbb Q$ and let $\frac{p_n}{q_n}$ be its $n$-th convergent (where $p_n$ and $q_n$ are coprime).
Then either $$\left|x-\frac{p_n}{q_n}\right|<\frac{1}{2q_n^2} \ \text{or} \ \left|x-\frac{p_{n+1}}{q_{n+1}}\right|<\frac{1}{2q_{n+1}^2}.$$
\end{theorem}
And in a certain way, only continued fractions get very close to irrational numbers.
\begin{theorem}[Legendre]
Let $x \in \mathbb R \setminus \mathbb Q$ and let $\frac{p}{q}$ with $p,q \in \mathbb Z$ satisfy
$|x-\frac{p}{q}|<\frac{1}{2q^2}$. Then $\frac{p}{q}$ is a~convergent of $x$.
\end{theorem}
\section{Continued fractions and continuants}
The convergents of continued fractions are closely related to the so-called {\em continuants} $K_n(x_1,\dots, x_n)$.
\begin{theorem}\label{CF_polynomials}
Let $x_0 \in \mathbb R, \ x_i >0,\ i \in \mathbb N$. Then it holds
$$x_0+\cfrac{1}{x_1+\cfrac{1}{{\ddots+\cfrac{1}{x_n}}}}=\frac{K_{n+1}(x_0,x_1,\dots, x_n)}{K_n(x_1,\dots,x_n)},$$
where the polynomial $K_{n}(x_1,\dots, x_n)$ is given by the recurrence relation $K_{-1}=0, \ K_0=1$ and for $n\geq 1$
$K_{n}(x_1,\dots,x_n)=K_{n-2}(x_1,\dots, x_{n-2})+x_nK_{n-1}(x_1,\dots, x_{n-1})$.
\end{theorem}

\begin{theorem}\label{symetrie}
For every $n \in \mathbb N$ and $x_1,\dots, x_n \in \mathbb R$, we have
$$K_n(x_1,\dots, x_n)=K_n(x_n,\dots, x_1).$$
\end{theorem}

\begin{corollary}
Let $[a_0,\dots, a_n,\dots]$ be the continued fraction of an irrational number $x$.
Then its $n$-th convergent $\frac{p_n}{q_n}$ satisfies
$$p_n=K_{n+1}(a_0,\dots,a_n), \quad q_n=K_n(x_1,\dots,x_n).$$
\end{corollary}
\section{Continued fractions of quadratic numbers}
We will call a~{\em quadratic irrational} an irrational root $\alpha$ of a~quadratic equation
$$Ax^2+Bx+C=0,$$
where $A,B,C \in \mathbb Z$.
The second root of the equation will be denoted $\alpha'$ and called the {\em (algebraic) conjugate} of $\alpha$.

In order to state the theorem describing continued fractions of quadratic irrationals, we need to recall that a~continued fraction $[a_0,\dots, a_n,\dots]$ is called {\em eventually periodic} if $[a_0,\dots, a_n,\dots]=[a_0,\dots, a_k, \overline{a_{k+1},\dots, a_{\ell}}]$ starts with a~preperiod $a_0,\dots, a_k$ and then a~period $a_{k+1},\dots, a_{\ell}$ is repeated an infinite number of times.
\begin{theorem}[Lagrange]
Let $\alpha \in \mathbb R \setminus \mathbb Q$. The continued fraction of $\alpha$ is eventually periodic if and only if $\alpha$ is a~quadratic irrational.
\end{theorem}

\begin{theorem}[Galois]\label{Galois}
Let $\alpha$ be a~quadratic irrational and $\alpha'$ its conjugate.
The continued fraction of $\alpha$ is purely periodic if and only if $\alpha>1$ and $\alpha' \in (-1,0)$.
\end{theorem}

\begin{example}
Let $\alpha=\frac{1+\sqrt{5}}{2}$, i.e., the so-called {\em Golden ratio}, then it is the root of $x^2-x-1=0$ and
$\alpha'=\frac{1-\sqrt{5}}{2}\in (-1,0)$. The continued fraction of $\alpha$ is indeed purely periodic since
$$\alpha=1+\frac{-1+\sqrt{5}}{2}=1+\frac{1}{\frac{1+\sqrt{5}}{2}}=1+\frac{1}{\alpha},$$
consequently $\alpha=[\ \overline{1}\ ]$.
\end{example}
In the sequel when we restrict our consideration to square roots of natural numbers, we will make use of the following lemma.
\begin{lemma}\label{rozvoj_pozpatku}
Let $\alpha$ be a~quadratic irrational and $\alpha'$ its conjugate.
If $\alpha$ has a~purely periodic continued fraction $[\overline{a_0,a_1,\dots,a_n}]$,
then $\frac{-1}{\alpha'}=[\overline{a_n,\dots,a_1,a_0}]$.
\end{lemma}

\section{Continued fractions of $\sqrt{N}$}
Let us consider $N \in \mathbb N\setminus \{0\}$ such that $N$ is not a~square.
If $N=k^2$ for some $k \in \mathbb N$, then $\sqrt{N}=k$ and the continued fraction is $\sqrt{N}=[k]$.
For every $N\in \mathbb N\setminus \{0\}$ which is not a~square, there exists a~unique $n\in \mathbb N\setminus \{0\}$ and a~unique $j \in \{1,\dots, 2n\}$ such that $N=n^2+j$.
\begin{theorem}\label{palindrom}
For every $n \in \mathbb N\setminus \{0\}$ and every $j \in \{1,\dots, 2n\}$
the continued fraction of $\sqrt{n^2+j}$ is of the form $[n,\overline{a_1,\dots,a_r, 2n}]$, where $a_1\dots a_r$ is a~palindrome.
\end{theorem}
\begin{proof}
This proof can be found in~\cite{La} page 15. However we repeat it here since it follows almost immediately from the previous statements and it gives an insight into the form of continued fractions of quadratic numbers.
Denote $\alpha=n+\sqrt{n^2+j}$. Then $\alpha$ is a~quadratic irrational greater than $1$ and $\alpha'=n-\sqrt{n^2+j}\in (-1,0)$.
Therefore $\alpha$ has by Theorem~\ref{Galois} a~purely periodic continued fraction, i.e., there exist $a_1, \dots, a_r \in \mathbb N$ such that $\alpha=[\overline{2n,a_1,\dots, a_r}]$. It is thus evident that $\sqrt{n^2+j}=[n,\overline{a_1,\dots, a_r,2n}]$.
It remains to prove that $a_1\dots a_r$ is a~palindrome. According to Lemma~\ref{rozvoj_pozpatku} the number $\frac{-1}{\alpha'}$ has its continued fraction equal to $[\overline{a_r,\dots, a_1, 2n}]$.
We obtain thus
$$\sqrt{n^2+j}=n+\cfrac{1}{\cfrac{-1}{n-\sqrt{n^2+j}}}=n+\cfrac{1}{\cfrac{-1}{\alpha'}}=[n,\overline{a_r,\dots, a_1, 2n}].$$
Since the continued fraction of irrational numbers is unique and we have
$$\sqrt{n^2+j}=[n,\overline{a_1,\dots, a_r,2n}]=[n,\overline{a_r,\dots, a_1, 2n}],$$
it follows that $a_1=a_r, \ a_2=a_{r-1}$ etc. Consequently, $a_1\dots a_r$ is a~palindrome.
\end{proof}
\begin{theorem}
The continued fraction of the form $[n,\overline{a_1,\dots,a_r, 2n}]$, where $a_1\dots a_r$ is a~palindrome, corresponds to $\sqrt{N}$ for a~rational number $N$.
\end{theorem}
\begin{proof}
Denote by $x$ the number whose continued fraction equals $[n,\overline{a_1,\dots,a_r, 2n}]$, i.e.,
$$x=n+\cfrac{1}{a_1+\cfrac{1}{\ddots+\cfrac{1}{a_r+\cfrac{1}{2n+(x-n)}}}}\ .$$
Hence by Theorem~\ref{CF_polynomials}, $$\begin{array}{rcl}x-n&=&\frac{K_{r}(a_2,\dots, a_r,x+n)}{K_{r+1}(a_1,\dots,a_r,x+n)}, \\
&=&\frac{K_{r-2}(a_2,\dots, a_{r-1})+(x+n)K_{r-1}(a_2,\dots, a_r)}{K_{r-1}(a_1,\dots, a_{r-1})+(x+n)K_{r}(a_1,\dots, a_r)}.\end{array}$$
By Theorem~\ref{symetrie} and since $a_1\dots a_r$ is a~palindrome, we have $K_{r-1}(a_1,\dots, a_{r-1})=K_{r-1}(a_2,\dots,a_r)$.
Consequently, we obtain
$$x=\sqrt{n^2+\frac{2nK_{r-1}(a_1,\dots, a_{r-1})+K_{r-2}(a_2,\dots,a_{r-1})}{K_{r}(a_1,\dots, a_r)}},$$
where under the square root, there is certainly a~rational number since by their definition, continuants with integer variables are integers.
\end{proof}
In the sequel, let us study the length of the period and the form of the continued fraction of $\sqrt{N}=\sqrt{n^2+j}$ in dependence on $n$ and $j$, where $n \in \mathbb N\setminus \{0\}$ and $j \in \{1,\dots, 2n\}$. We will prove only some of observations since the proofs are quite technical and space-demanding.
The rest of proofs may be found in~\cite{SOC}. In Table~\ref{tabulka_hotovo}, we have highlighted all classes of $n$ and $j$ for which their continued fractions of $\sqrt{N}=\sqrt{n^2+j}$ have been described.

\begin{observation}
The continued fraction of $\sqrt{N}$ has period of length $1$ if and only if $N=n^2+1$. It holds then
$\sqrt{N}=[n,\overline{2n}]$.
\end{observation}
\begin{proof}
This observation has been done already in~\cite{Si}.\\
\noindent $(\Leftarrow):$
$$\begin{array}{l}\sqrt{n^2+1}=n+\cfrac{\sqrt{n^2+1}-n}{1}=\\
=n+\cfrac{1}{\sqrt{n^2+1}+n}=n+\cfrac{1}{2n+\cfrac{\sqrt{n^2+1}-n}{1}},
\end{array}$$
hence $\sqrt{N}=[n,\overline{2n}]$.\\
\noindent $(\Rightarrow):$
If the length of the period equals $1$, then by Theorem~\ref{palindrom} we have $\sqrt{N}=[n,\overline{2n}]$.
 $$ \sqrt{n^2+j}=n+(\sqrt{n^2+j}-n)=n+\cfrac{1}{2n+\sqrt{n^2+j}-n},$$
 hence we have
 $$\begin{array}{rcl}\sqrt{n^2+j}-n&=&\cfrac{1}{2n+\sqrt{n^2+j}-n},\\
 \cfrac{j}{\sqrt{n^2+j}+n}&=&\cfrac{1}{\sqrt{n^2+j}+n},\\
 j&=&1.
 \end{array}$$
\end{proof}
\begin{observation}\label{per2}
The continued fraction of $\sqrt{N}$ has period of length $2$ if and only if $\frac{2n}{j}$ is an integer. It holds then
$\sqrt{N}=[n,\overline{\frac{2n}{j},2n}]$.
\end{observation}
\begin{proof}
\noindent $(\Leftarrow):$
$$\begin{array}{rcl}\sqrt{n^2+j}&=&n+(\sqrt{n^2+j}-n)\\
&=&n+\cfrac{j}{\sqrt{n^2+j}+n}\\
&=&n+\cfrac{1}
{\cfrac{2n}{j}+\cfrac{\sqrt{n^2+j}-n}{j}},\\
&=&n+\cfrac{1}
{\cfrac{2n}{j}+\cfrac{1}{\sqrt{n^2+j}+n}},\\
&=&n+\cfrac{1}
{\cfrac{2n}{j}+\cfrac{1}{2n+(\sqrt{n^2+j}-n)}},
\end{array}$$
thus $\sqrt{N}=[n,\overline{\frac{2n}{j},2n}]$.\\
\noindent $(\Rightarrow):$ If the length of the period equals $2$, then by Theorem~\ref{palindrom} we have $\sqrt{N}=[n,\overline{x,2n}]$.
$$\begin{array}{rcl}\sqrt{n^2+j}&=&n+(\sqrt{n^2+j}-n)\\
&=&n+\cfrac{1}{x+\cfrac{1}{2n+(\sqrt{n^2+j}-n)}},\end{array}$$
hence we have
$$\begin{array}{rcl}
\sqrt{n^2+j}-n&=&\cfrac{1}{\cfrac{x(\sqrt{n^2+j}+n)+1}{\sqrt{n^2+j}+n}},\\
x&=&\frac{2n}{j}.
\end{array}$$
\end{proof}
\begin{observation}\label{per3}
If the continued fraction of $\sqrt{N}$ has period of length $3$, then $j$ is an odd number and $\sqrt{N}=[n,\overline{x,x,2n}]$, where $x$ is an even number.
\end{observation}
\begin{proof}
If the length of the period equals $3$, then by Theorem~\ref{palindrom} we have $\sqrt{N}=[n,\overline{x,x,2n}]$.
$$\sqrt{n^2+j}=n+\cfrac{1}{x+\cfrac{1}{x+\cfrac{1}{2n+(\sqrt{n^2+j}-n)}}},$$
hence we get
$j=\frac{2xn+1}{x^2+1}$.
The condition that $j$ and $x$ must be integers implies the statement.
\end{proof}

\begin{observation}
Let $j=4$. If $n$ is even, then the length of the period is $2$ and $\sqrt{N}=[n,\overline{\frac{2n}{j},2n}]$.
If $n$ is odd, then the length of the period is $5$ and $\sqrt{N}=[n, \overline{\frac{n-1}{2},1,1,\frac{n-1}{2},2n}]$.
\end{observation}
\begin{proof}
If $n$ is even, $\frac{2n}{j}$ is an integer and the statement is a~corollary of Observation~\ref{per2}. If $n$ is odd, it holds
$$\sqrt{n^2+4}=$$
$$\begin{array}{c}
=n+(\sqrt{n^2+4}-n)\\
=n+\cfrac{1}{\cfrac{\sqrt{n^2+4}+n}{4}}\\
=n+\cfrac{1}{\cfrac{2n-2}{4}+\cfrac{\sqrt{n^2+4}-(n-2)}{4}}\\
=n+\cfrac{1}{\cfrac{n-1}{2}+\cfrac{1}{\cfrac{4(\sqrt{n^2+4}+(n-2))}{4n}}}\\
=n+\cfrac{1}{\cfrac{n-1}{2}+\cfrac{1}{1+\cfrac{\sqrt{n^2+4}-2}{n}}}\\
=n+\cfrac{1}{\cfrac{n-1}{2}+\cfrac{1}{1+\cfrac{1}{\cfrac{\sqrt{n^2+4}+2}{n}}}}\\
=n+\cfrac{1}{\cfrac{n-1}{2}+\cfrac{1}{1+\cfrac{1}{1+\cfrac{\sqrt{n^2+4}-(n-2)}{n}}}}\\
=n+\cfrac{1}{\cfrac{n-1}{2}+\cfrac{1}{1+\cfrac{1}{1+\cfrac{1}{\cfrac{n(\sqrt{n^2+4}+n-2)}{4n}}}}}\\
=n+\cfrac{1}{\cfrac{n-1}{2}+\cfrac{1}{1+\cfrac{1}{1+\cfrac{1}{\cfrac{2n-2}{4}+\cfrac{\sqrt{n^2+4}-n}{4}}}}}\\
=n+\cfrac{1}{\cfrac{n-1}{2}+\cfrac{1}{1+\cfrac{1}{1+\cfrac{1}{\cfrac{n-1}{2}+\cfrac{1}{2n+(\sqrt{n^2+4}-n)}}}}},
\end{array}$$
thus $\sqrt{N}=[n, \overline{\frac{n-1}{2},1,1,\frac{n-1}{2},2n}]$.
\end{proof}

\begin{observation}
For $n> 1$ and $j=2n-1$ the length of the period is $4$ and the continued fraction is then $\sqrt{N}=[n,\overline{1,n-1,1,2n}]$.
\end{observation}
\begin{proof}
$$\sqrt{n^2+2n-1}=$$
$$\begin{array}{c}
=n+(\sqrt{n^2+2n-1}-n)\\
=n+\cfrac{1}{\cfrac{\sqrt{n^2+2n-1}+n}{2n-1}}\\
=n+\cfrac{1}{\cfrac{2n-1}{2n-1}+\cfrac{\sqrt{n^2+2n-1}-(n-1)}{2n-1}}\\
=n+\cfrac{1}{1+\cfrac{1}{\cfrac{\sqrt{n^2+2n-1}+(n-1)}{2}}}\\
=n+\cfrac{1}{1+\cfrac{1}{\cfrac{2n-2}{2}+\cfrac{\sqrt{n^2+2n-1}-(n-1)}{2}}}\\
=n+\cfrac{1}{1+\cfrac{1}{n-1+\cfrac{1}{\cfrac{\sqrt{n^2+2n-1}+(n-1)}{2n-1}}}}\\
=n+\cfrac{1}{1+\cfrac{1}{n-1+\cfrac{1}{\cfrac{2n-1}{2n-1}+\cfrac{\sqrt{n^2+2n-1}-n}{2n-1}}}}\\
=n+\cfrac{1}{1+\cfrac{1}{n-1+\cfrac{1}{1+\cfrac{1}{2n+(\sqrt{n^2+2n-1}-n)}}}},\end{array}$$
hence $\sqrt{N}=[n,\overline{1,n-1,1,2n}]$
\end{proof}

\begin{observation} For $n>3$ and $j=2n-3$, either the length of the period is $4$ if $n$ is odd and the continued fraction is then $\sqrt{N}=[n,\overline{1,\frac{n-3}{2},1,2n}]$, or the length of the period is $6$ if $n$ is even and the continued fraction is then $\sqrt{N}=[n,\overline{1,\frac{n}{2}-1,2,\frac{n}{2}-1,1,2n}]$.
\end{observation}
\begin{proof}
\noindent For $n$ odd:
$$\sqrt{n^2+2n-3}=$$
$$\begin{array}{c}
=n+(\sqrt{n^2+2n-3}-n)\\
=n+\cfrac{1}{\cfrac{\sqrt{n^2+2n-3}+n}{2n-3}}\\
=n+\cfrac{1}{\cfrac{2n-3}{2n-3}+\cfrac{\sqrt{n^2+2n-3}-(n-3)}{2n-3}}\\
=n+\cfrac{1}{1+\cfrac{1}{\cfrac{\sqrt{n^2+2n-3}+(n-3)}{4}}}\\
=n+\cfrac{1}{1+\cfrac{1}{\cfrac{2n-6}{4}+\cfrac{\sqrt{n^2+2n-3}-(n-3)}{4}}}\\
=n+\cfrac{1}{1+\cfrac{1}{\cfrac{n-3}{2}+\cfrac{1}{\cfrac{\sqrt{n^2+2n-3}+(n-3)}{2n-3}}}}\\
=n+\cfrac{1}{1+\cfrac{1}{\cfrac{n-3}{2}+\cfrac{1}{1+\cfrac{\sqrt{n^2+2n-3}-n}{2n-3}}}}\\
=n+\cfrac{1}{1+\cfrac{1}{\cfrac{n-3}{2}+\cfrac{1}{1+\cfrac{1}{2n+(\sqrt{n^2+2n-3}-n)}}}},\end{array}$$
thus $\sqrt{N}=[n,\overline{1,\frac{n-3}{2},1,2n}]$.

\noindent For $n$ even:
$$\sqrt{n^2+2n-3}=$$
$$\begin{array}{c}
=n+(\sqrt{n^2+2n-3}-n)\\
=n+\cfrac{1}{\cfrac{\sqrt{n^2+2n-3}+n}{2n-3}}\\
=n+\cfrac{1}{\cfrac{2n-3}{2n-3}+\cfrac{\sqrt{n^2+2n-3}-(n-3)}{2n-3}}\\
=n+\cfrac{1}{1+\cfrac{1}{\cfrac{\sqrt{n^2+2n-3}+(n-3)}{4}}}\\
=n+\cfrac{1}{1+\cfrac{1}{\cfrac{2n-4}{4}+\cfrac{\sqrt{n^2+2n-3}-(n-1)}{4}}}\\
=n+\cfrac{1}{1+\cfrac{1}{\cfrac{n}{2}-1+\cfrac{1}{\cfrac{\sqrt{n^2+2n-3}+(n-1)}{n-1}}}}\\
=n+\cfrac{1}{1+\cfrac{1}{\cfrac{n}{2}-1+\cfrac{1}{\cfrac{2n-2}{n-1}+\cfrac{\sqrt{n^2+2n-3}-(n-1)}{n-1}}}}\\
=n+\cfrac{1}{1+\cfrac{1}{\cfrac{n}{2}-1+\cfrac{1}{2+\cfrac{1}{\cfrac{\sqrt{n^2+2n-3}+(n-1)}{4}}}}}\\
=n+\cfrac{1}{1+\cfrac{1}{\cfrac{n}{2}-1+\cfrac{1}{2+\cfrac{1}{\cfrac{2n-4}{4}+\cfrac{\sqrt{n^2+2n-3}-(n-3)}{4}}}}}\\
=n+\cfrac{1}{1+\cfrac{1}{\cfrac{n}{2}-1+\cfrac{1}{2+\cfrac{1}{\cfrac{n}{2}-1+\cfrac{1}{\cfrac{\sqrt{n^2+2n-3}+(n-3)}{2n-3}}}}}}
\end{array}$$
$$\begin{array}{c}
=n+\cfrac{1}{1+\cfrac{1}{\cfrac{n}{2}-1+\cfrac{1}{2+\cfrac{1}{\cfrac{n}{2}-1+\cfrac{1}{\cfrac{2n-3}{2n-3}+\cfrac{\sqrt{n^2+2n-3}-n}{2n-3}}}}}}\\
=n+\cfrac{1}{1+\cfrac{1}{\cfrac{n}{2}-1+\cfrac{1}{2+\cfrac{1}{\cfrac{n}{2}-1+\cfrac{1}{1+\cfrac{1}{2n+(\sqrt{n^2+2n-3}-n)}}}}}}.\end{array}$$
Thus $\sqrt{N}=[n,\overline{1,\frac{n}{2}-1,2,\frac{n}{2}-1,1,2n}]$.

\end{proof}

\begin{observation} Let $k\in \mathbb N$. Let us summarize lengths $\ell$ of periods and the form of continued fractions for several classes (described in an analogous way) of $n$ and $j$.

\begin{tabular}{c|c|c|c}
$n$&$j$&$\ell$& $\sqrt{N}$\\
\hline
\hline
$5k+1, \ k\geq 1$& $\frac{4n+1}{5}$&$3$&$[n,\overline{2,2,2n}]$\\
\hline
$6k+5$& $\frac{2n-1}{3}$&$4$&$[n,\overline{3,\frac{n-1}{2},3,2n}]$\\
\hline
$9k+4, \ k\geq 1$&$n-2$ &$4$&$[n,\overline{2,\frac{2n-8}{9},2,2n}]$\\
\hline
$5k+4$& $\frac{8n+3}{5}$&$4$&$[n,\overline{1,3,1,2n}]$\\
\hline
$3k+2$& $\frac{5n+2}{3}$&$4$&$[n,\overline{1,4,1,2n}]$
\\
\hline
$3k+2, \ k\geq 1$&$2n-2$&$4$&$[n,\overline{1,\frac{2n-4}{3},1,2n}]$\\
\hline
$3k+2$&$\frac{4n+1}{3}$&$4$&$[n,\overline{1,1,1,2n}]$\\
\hline
$2k+1, \ k\geq 1$&$\frac{3n+1}{2}$&$4$&$[n,\overline{1,2,1,2n}]$\\
\hline
$5k+2, \ k\geq 1$&$n-1$&$4$&$[n,\overline{2,\frac{2n-4}{5},2,2n}]$\\
\hline
$6k+1, \ k\geq 1$&$\frac{5n+1}{6}$&$4$&$[n,\overline{2,2,2,2n}]$\\
\hline
$5k+3$&$\frac{6n+2}{5}$&$5$&$[n,\overline{1,1,1,1,2n}]$\\
\hline
$10k+7$&$\frac{2n+1}{5}$&$6$&$[n,\overline{4,1,\frac{n-3}{2},1,4,2n}]$\\
\hline
$3k+1, \ k\geq 1$&$\frac{2n+1}{3}$&$6$&$[n,\overline{2,1,n-1,1,2,2n}]$\\
\hline
$3k+1, \ k\geq 1$&$n+1$&$6$&$[n,\overline{1,1,\frac{2n-2}{3},1,1,2n}]$\\
\hline
$7k+3, \ k\geq 1$&$n+2$&$6$&$[n,\overline{1,1,\frac{2n-6}{7},1,1,2n}]$\\
\hline
$6k+4$&$\frac{7n+2}{6}$&$6$&$[n,\overline{1,1,2,1,1,2n}]$\\
\hline
$3k+1, \ k\geq 1$&$\frac{4n+2}{3}$&$6$&$[n,\overline{1,2,n,2,1,2n}]$\\
\hline
$7k+5$&$\frac{8n+2}{7}$&$8$&$[n,\overline{1,1,3,n,3,1,1,2n}]$\\
\hline
$6k+2, \ k\geq 1$&$\frac{2n-1}{3}$ &$8$&$[n,\overline{3,\frac{n-2}{2},1,4,1,\frac{n-2}{2}, 3,2n}]$\\

\end{tabular}
\end{observation}

\begin{observation}
If the period of the continued fraction of $\sqrt{N}=\sqrt{n^2+j}$ contains $p\geq 1$ ones as its palindromic part, i.e., $\sqrt{N}=[n,\overline{\underbrace{1,\dots,1}_p,2n}]$ then $p+1 \not =3k$, where $k\in \mathbb N$, and $j=\frac{2nF_{p-1}+F_{p-2}}{F_{p}}$, where $F_n$ denotes the $n$-th Fibonacci number given by the recurrence relation $F_{-1}=0, \ F_0=1$ and $F_{n}=F_{n-2}+F_{n-1}$ for all $n \geq 1$.
\end{observation}

We made also one observation that turned out to be false.
\begin{observation} For $\sqrt{N}$ the length of the period of the continued fraction is less than or equal to $2n$.
\end{observation}
This observation was made when contemplating a~table of periods of $\sqrt{N}$ for $N \leq 1000$.
However, in~\cite{Hi} it is shown that for $N=1726$ with $n=41$, the period of the continued fraction of $\sqrt{N}$ is of length $88 > 82=2n$.
A~rougher upper bound comes from~\cite{Si}.
\begin{theorem}
For $\sqrt{N}$ the length of the period of the continued fraction is less than or equal to $2N$.
\end{theorem}
Let us terminate with two observations that have not been proved yet.
\begin{observation}
No element of the period of $\sqrt{N}$ apart from the last one is bigger than $n$.
\end{observation}
\begin{observation}
There is no period of an odd length for $j=4k+3$, where $k\in \mathbb N$.
\end{observation}

\thanks 
We acknowledge financial support by the Czech Science Foundation
grant 201/09/0584, by the grants MSM6840770039 and LC06002 of the
Ministry of Education, Youth, and Sports of the Czech Republic,
and by the Grant Agency of the Czech Technical
University in Prague, grant No. SGS11/162/OHK4/3T/14.



\begin{figure}[!h]
\begin{center}
\resizebox{10 cm}{!}{\includegraphics{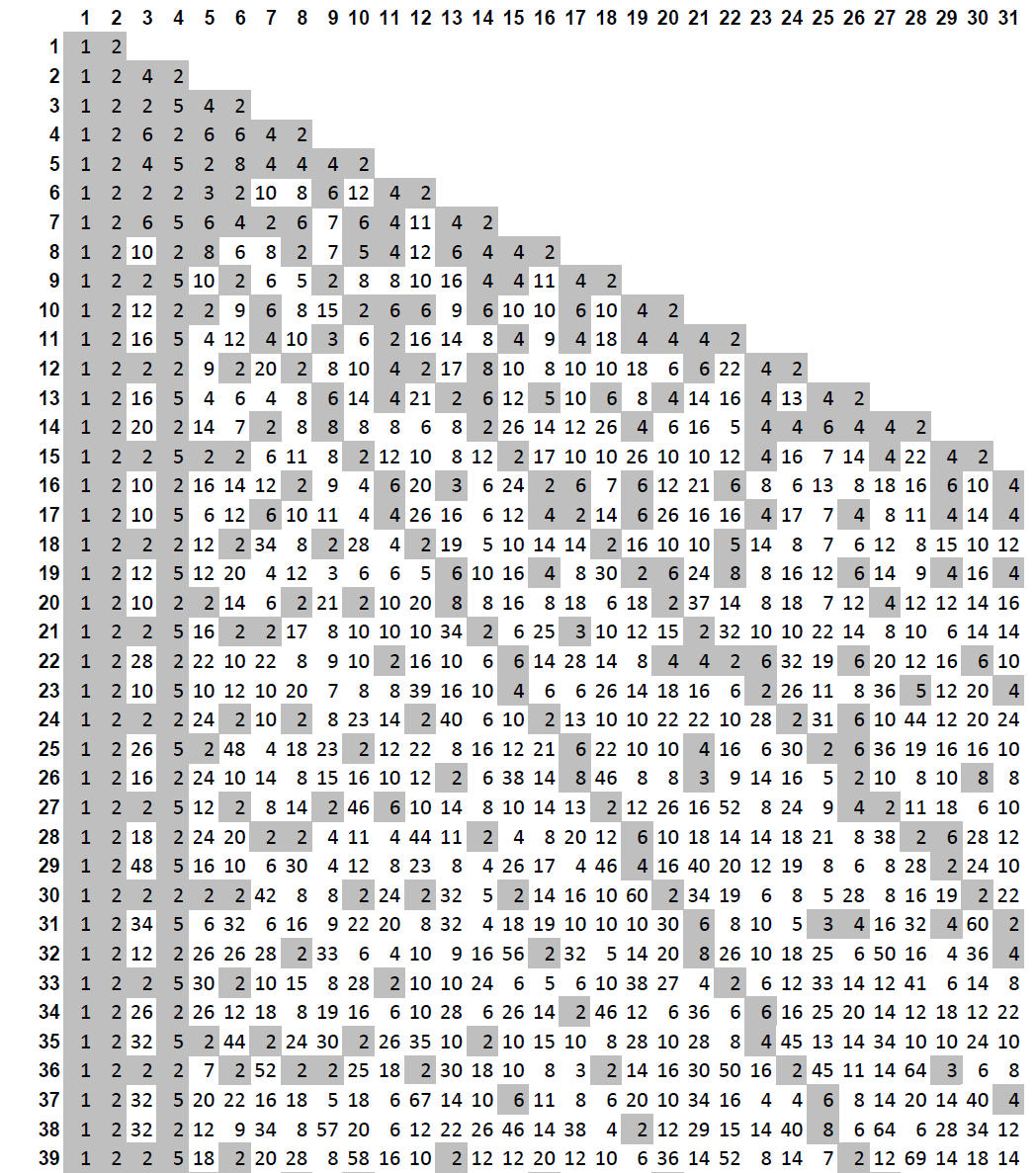}}
\caption{All classes of $n\leq 39$ (first column) and $j\leq 31$ (first row) for which their continued fractions of $\sqrt{N}=\sqrt{n^2+j}$ have been described are highlighted.}
\end{center}
\label{tabulka_hotovo}
\end{figure}
\end{document}